    \documentclass[reqno]{amsproc}
    \usepackage{amsmath,amssymb,amsthm,amsfonts,latexsym}

    \theoremstyle{plain}
   \newtheorem{thm}{Theorem}
   \newtheorem{pro}[thm]{Proposition}
   
   \newtheorem{cor}[thm]{Corollary}

   \theoremstyle{definition}

   \theoremstyle{remark}
   \newtheorem{rem}[thm]{{\it Remark}}
   \newtheorem*{com}{{\it Comments}}
   
    \newtheorem*{note}{{Note}}
   \newtheorem*{summ}{{Summing-up}}

   \def\tlabel{\label}

\DeclareMathOperator{\okr}{{\stackrel{{\scriptscriptstyle{\mathsf{def}}}}{=}}}
\DeclareMathOperator{\lin}{lin}
\DeclareMathOperator{\clolin}{clolin}
\def\dz#1{\mathcal D({#1})}
\def\funk#1#2#3{#1\colon#2\to#3}
\def\funkc#1#2#3#4#5{#1\colon#2\ni#3\mapsto#4\in#5}
\def\funkk#1#2#3#4{#1\ni#2\mapsto#3\in#4}
\def\Ge{\geqslant}
\def\gw{^*}
\def\is#1#2{\langle#1,#2\rangle}
\def\iso#1#2{\left\langle#1,#2\right\rangle}
\def\isp{\langle\;\cdot\;,-\rangle}
\def\kolo#1{\textcircled{$\scriptstyle{#1}$}}
\def\Le{\leqslant}
\def\liczp#1{{${#1}^{\text {\rm o}}$}}
\def\taggw{$\boldsymbol{\ast}$}
\def\wynik{$\implies$}
\def\zb#1#2{\{{#1}\colon\ {#2}\}}
\newcommand*\ddc{\mathcal D}
\newcommand*\eec{\mathcal E}
\newcommand*\hhc{\mathcal H}
\newcommand*\ccb{\mathbb C}
\newcommand*\nnb{\mathbb N}
\newcommand*\zzb{\mathbb Z}

\usepackage{color}
\definecolor{Light}{gray}{.75}

   \def\alga{{\sf A}}

   \def\algaa{{\sf a}}
   \def\algbb{{\sf b}}
   \def\sem{{\sf S}}
   \def\sems{{\sf s}}
   \def\semt{{\sf t}}
\begin{document}

   %%%%%%%%%%%

   \title[Positive definite kernels and Hilbert
   C${}^*$--modules ]{Murphy's {\em Positive definite kernels and
   Hilbert C${}^*$--modules} reorganized}
   \author[F.H. Szafraniec]{Franciszek Hugon Szafraniec}
   \address{Instytut        Matematyki,         Uniwersytet
   Jagiello\'nski, ul. \L ojasiewicza 6, 30 348 Krak\'ow, Poland}
   \email{franciszek.szafraniec@im.uj.edu.pl}
   \thanks{This work was  supported  by the MNiSzW grant
N201 026 32/1350 and at some stage also by the Dean of the Faculty
of Mathematics and Computer Science of the Jagiellonian University
research grant 23K/ZBW/000073.}
   %\dedicatory{}
   %%%%%%%%%%%%%%%%%%
   \begin{abstract}
   The paper the title refers to is that in {\em Proceedings of
   the Edinburgh Mathematical Society}, {\bf 40}\,(1997), 367-374.
   Taking it as an excuse we intend to realize a twofold purpose:
   \begin{enumerate}
   \item[\liczp 1] to atomize that important result showing by the
   way connections which are out of favour, \item[\liczp 2] to
   rectify a tiny piece of history.
   \end{enumerate}
   The objective \liczp 1 is going to be achieved by adopting
means adequate to goals; it is of great gravity and this is just
Mathematics. The other, \liczp 2, comes from author's internal
need of showing how ethical values in Mathematics are getting
depreciated. The latter have nothing to do with the previous
issue; the coincidence is totally accidental.
   \end{abstract}
   \maketitle

   \section*{Reproducing kernel Hilbert C${}^*$--modules}
   \subsection*{Rudiments of the theory of Hilbert C${}^*$--modules}
    Let $\alga$ be a C${}^*$--algebra with its norm
denoted\,\footnote{\;Warning: we do not copy slavishly the
notation of \cite{lance} though this is a standard monograph of
the subject. In order to protect reader's head to be in a whirl we
put a subscript in notation of norms; inner products are less
dangerous.} by $\|\,\cdot\,\|_\alga$. An {\em inner product
$\alga$--module} is a \underbar{right} $\alga$--module $\varXi$
(with scalar multiplication compatible with this in $\varXi$ as
well as that in $\alga$) with a mapping
$$\varXi\times \varXi\ni(\xi,\eta)\longmapsto \is \xi\eta\,\in\alga$$ such that
   \begin{enumerate}
   \item[(a)] it is $\ccb$--linear in the \underbar{second}
variable; \item[{(b)}] $\is \xi{\eta\algaa}=\is \xi\eta\algaa$,
$\xi,\eta\in \varXi$, ${\sf a}\in\alga$;
   \item[{(c)}] $\is \eta\xi=\is \xi\eta^*$, $\xi,\eta\in \varXi$;
\item[{(d)}] $\is \xi\xi\Ge 0$; if $\is \xi\xi=0$ then $\xi=0$.
   \end{enumerate}
   If the second condition in (d) is dropped we call $\varXi$ {\em
semi--inner product $\alga$--module}.

   \begin{pro}[Proposition 2.3 in \cite{bill_p}]\tlabel{t1.31.1}
   Conditions {\rm (a)--(d)} above imply that
$\|\xi\|_\varXi\okr\sqrt{\|\is \xi\xi\|_\alga}$ is a norm on
$\varXi$ and\,\footnote{\;The proof in \cite{bill} caries over to
the case of semi--inner product as well, for another look at
\cite{lance}, p. 3.}
   \begin{enumerate}
   \item[(i)]
$\|\xi\algaa\|_\varXi\Le\|\xi\|_\varXi\|\algaa\|_\alga$,
$\xi\in\varXi$, $\algaa\in\alga$;
   \item[(ii)] $\is \eta\xi\is \xi\eta\Le\|\eta\|_\varXi^2\is
\xi\xi$, $\xi,\eta\in\varXi$; \item[(iii)] $\|\is
\xi\eta\|_\alga\Le
   \|\xi\|_\varXi\,\|\eta\|_\varXi$, $\xi,\eta\in\varXi$.
   \end{enumerate}
   \end{pro}

 If $\varXi$ is complete with respect to the norm
$\|\,\cdot\,\|_\varXi$, it is called a {\em Hilbert
$\alga$--module} (it belongs to the category of C${}\gw$--modules
if one wants to hide \alga).

   A Hilbert $\alga$--module is a Banach space and for that reason
one can consider \underbar{bounded} linear operators between two
such spaces, $\varXi$ and $\varXi_1$ say; denote the totality of
those by ${\bf B}(\varXi,\varXi_1)$. Call a map $T\colon \varXi\to
\varXi_1$ {\em adjointable} if there is another map $T^*\colon
\varXi_1\to \varXi$ such that
   \begin{equation*}
   \is{T\xi}{\xi_1}=\is {\xi}{T^* \xi_1},\quad \xi\in
\varXi,\;\xi_1\in\varXi_1.
   \end{equation*}
   If $T$ is adjointable then it must necessarily be
$\ccb$--linear\,\footnote{\;'Linear' or in abbreviation '$\lin$'
always refers to $\ccb$--linearity. If $\alga$--has a unit, it is
needless to think separately of linearity when $\alga$--linearity
is around.} as well as $\alga$--linear, and, due to
Banach--Steinhaus theorem also bounded. Denote by ${\bf
B}^*(\varXi,\varXi_1)$ the set of all adjointable operators;
apparently ${\bf B}^*(\varXi,\varXi_1)\subset{\bf
B}(\varXi,\varXi_1)$. For further use we make a shorthand notation
${\bf B}^*(\varXi)\okr{\bf B}^*(\varXi,\varXi)$; this, with the
involution ${}^*$, is a C${}\gw$--algebra. Sometimes we may have a
need to get the C${}^*$--algebra involved in the notation; we just
put the C${}^*$--algebra in the subscript like in ${\bf
B}^*_\alga(\varXi)$.

For $\xi\in\varXi$ and $\xi_1\in\varXi_1$ define the operators
$T_{\xi,\eta}$ by and $T_{\xi}$ by
   \begin{gather}
   \begin{split}\label{2.31.1}
   T_{\xi,\xi_1}\eta\okr \xi_1\is \xi\eta\\ T_{\xi}\eta\okr\is
\xi\eta\phantom{ii}
   \end{split},\quad\quad  \eta\in\varXi.
   \end{gather}
Then $T_{\xi,\xi_1}$'s belong to ${\bf B}^*(\varXi,\varXi_1)$ and
$T_\xi$'s do to ${\bf B}^*(\varXi,\alga)$. Moreover,
$T_{\xi,\algaa}=T_{\xi\algaa^*}$. Let ${\bf K}(\varXi,\varXi_1)$
stand for the closed (in ${\bf B}(\varXi,\varXi_1)$) linear span
of all the $T_{\xi,\xi_1}$'s. It is clear that $T_{\xi,\xi_1}$ as
well as $T_\xi$ is $\alga$--linear.

Notice that $\alga$ itself is a Hilbert $\alga$--module with the
inner product
   \begin{equation*} %\label{3.30.1}
   \is ab_\alga\okr a^*b,\quad a,b\in\alga
   \end{equation*}
   and the {\alga}--module norm coincides with that of the
C${}^*$--algebra $\alga$. Another thing which is worthy to mention
is that ${\bf B}^*(\alga)$ is isomorphic to $\alga$ itself, cf.
\cite{lance} p.10.
   \begin{com}
   If $\alga$ has a unit ${\sf e}$ then ${\bf
K}(\varXi,\alga)={\bf B}^*(\varXi,\alga)$, cf. \cite{lance} p.13.
This is so because $T\in{\bf B}^*(\varXi,\alga)$ is of the form
$T=T_{T\gw{{\sf e}}}$, which as such belongs to ${\bf
K}(\varXi,\alga)$. Furthermore, ${\bf K}(\varXi,\alga)$ is
precisely the set of all those bounded $\alga$-linear mappings
which are Riesz representable through an $\alga$--inner product.

   On the other hand, the Riesz representation theorem does not
apply to the members of ${\bf B}(\varXi,\alga)$; if it does
$\varXi$ is called {\em self--dual} (more on self--dual Hilbert
C${}^*$--modules is in \cite{frank}).

   There is a substantial difference in behaviour of Hilbert
C${}^*$--modules comparing to Hilbert spaces: the
orthocomplemention does not perform as involution. However, one
useful tool remains:
   \begin{quote}
   if $\is\xi\eta=0$ for fixed $\xi$ and $\eta$ ranging over a
dense subset of $\varXi$ then $\xi=0$. It is so because the inner
product is continuous according to (iii) of Proposition
\ref{t1.31.1}.
   \end{quote}

   \end{com}

   \subsection*{C${}^*$--positive definite kernels}  Let $S$ be a
   set. Call a mapping $K\colon S\times S\to\alga$ a {\em
\alga--kernel on $S$} or shortly a kernel if no confusion arises.
An $\alga$--kernel on $S$ is said to be {\em $\alga$--positive
definite} (again, positive definite if no confusion arises) if
   \begin{equation} \label{1.25.1}
   \text{$\sum_{k,l}\algaa^*_kK(s_k,s_l)\algaa_l\Ge0$ for any
finite choice of $(s_n)_n\subset S$ and
$(\algaa_n)_n\subset\alga$.}
   \end{equation}
   Using the standard quadratic form (in two complex variables)
argument we get immediately get Hermitean symmetry
   \begin{equation} \label{5.10.2}
   K(s,t)={K(t,s)\gw},\quad s,t\in S.
   \end{equation}
   Two typical Schwarz inequalities can be derived from (ii) and
(iii) of Proposition \ref{t1.31.1} later.

  % \subsection*{Reproducing kernel Hilbert $\alga$--module:
%   the philosophy} So as to short exposition of the philosophy
%assume $\alga=\ccb$. Given a set $S$, then by a {\em reproducing
%kernel Hilbert space} (in short: RKHS) over $S$ we understand a
%\underbar{pair} $(K,\hhc)$ composed of a ($\ccb$-valued) kernel
%$K\colon S* S\to\ccb$ and a Hilbert space $\hhc$ of
%\underbar{functions} on $S$ such that the two decisive properties
%hold:
%\begin{gather*}
%   \{K_s\}_{s\in S}\subset \hhc,\\
%    f(x)=\is{f}{K_s},\quad f\in\hhc,\; s\in S
%   \end{gather*}
%   with $K_s\okr K(\,\cdot\,,s)$. Then the two basic facts:
%positive definiteness of $K$ and continuity of evaluation
%functionals follow. They are pretty often used in a more or less
%well done way as a definition of RKHS whereas any of these two are
%rather starting points of the construction. Another features of a
%RKHS follow from so balanced definition as well; we recommend to
%\cite{cue}, \cite{mult} and \cite{ksiazka} above all.
%
%
%
%   Now we come back to $\alga$ being an \underbar{arbitrary}
%C${}^*$--algebra.

   \subsection*{Reproducing kernel Hilbert $\alga$--module:
   the construction} Let $K$ be a \alga--positive definite kernel
on $S$. Set $K_s\okr K(s,\,\cdot\,)$ for the sections and
$$\ddc_K\okr\lin\zb{K_s\algaa}{s\in S,\,\algaa\in\alga}.$$ Therefore
the members of $\ddc_K$ are of the form ${\sum_iK_{s_i}\algaa_i}$,
which means they are mappings\,\footnote{\;It seems to be
suggestive to call mappings from $S$ to {\alga} just {\alga}--{\em
functions} on $S$.} from $S$ to {\alga}. Let us try to define an
\alga--inner product on $\ddc_K$ as follows:
   \begin{multline} \label{1.30.1}
\iso{\sum\nolimits_kK_{t_k}\algbb_k}{\sum\nolimits_iK_{s_i}\algaa_i}_K
\okr\sum\nolimits_{i,k}\algbb_k^*K(t_k,s_i)\algaa_i,\\
(s_m)_m,(t_n)_n\subset S,\;(\algaa_m)_m,(\algbb_n)_n\subset\alga .
   \end{multline}
   To see the inner product is well defined notice first that
${\sum_iK_{s_i}\algaa_i}=0$ forces
$\sum_{i,k}\algbb_k^*K(t_k,s_i)\algaa_i=0$ regardless what
$\sum_{k}K_{t_k}\algbb_k$ is. Then we get in a standard way that
   \begin{multline*}
\sum\nolimits_{i,k}\algbb_k^*K(t_k,s_i)\algaa_i=\sum\nolimits_{j,l}\algbb_l'
{}^*K(t_l',s_j')\algaa_i'
   \\\text{if\; $\sum\nolimits_iK_{s_i}\algaa_i= \sum\nolimits _jK_{s'{}_j}\algaa_j$\,
    and\,
$\sum\nolimits_kK_{t_k}\algbb_k=\sum\nolimits_lK_{t'{}_l}\algbb_l$,}
   \end{multline*}
   which proves the claim\,\footnote{\;Let us notice that this
simple observation is the key to the RKHS approach to be so
exciting. Usually people, even if they decide to follow the
construction up to the very end, at this point make the argument
 rather enigmatic if any at all. It is apparently needless to say
most of the RKHS--like constructions bear hallmarks of schism.}.

   The defining formula \eqref{1.30.1} turns into the reproducing
kernel property
   \begin{equation} \label{2.30.1}
   F(s)=\is F{K_s}_K,\quad s\in S,\;F\in\ddc_K.
   \end{equation}

    It is clear that $\ddc_K$ is an {\alga}--inner product module.
Now we want to complete it still having the completion to be a
Hilbert {\alga}--module composed of \underbar{{\alga}--functions}
on $S$\,\footnote{\;At this point the reproducing kernel space
idea was abandoned in \cite{murphy}}. For this let
$\widetilde\eec$ be an arbitrary Hilbert {\alga}--module in which
$\ddc_K$ is densely imbedded via the isometry\,\footnote{\;This is
in fact an {\alga}--isometry, that is it preserves {\alga}--inner
products, as a more careful look at the argument presented on p.
10 of \cite{lance} may ensure us.} $\funkk {\ddc}F{\tilde
F}{\widetilde\eec}$. For every $G\in\widetilde\eec$ the formula
$$F_G(s)\okr\is G{K_s}_{\widetilde\eec}$$ determines, by density
of $\ddc_K$ in $\widetilde\eec$, a unique {\alga}--function $F_G$
on $S$. It is a matter of straightforward calculation to check
that $\eec_K\okr\zb{F_G}{G\in\widetilde\eec}$ is a Hilbert
{\alga}--module of {\alga}--functions on $S$ with the inner
product being an `inverse image' of that in $\widetilde\eec$.
Needless to say that $\eec_K$ does not depend on a particular
choice of $\widetilde\eec$\,\footnote{\;Though
  elements of the RKHS approach can be traced on many occasions we
  would like to advertise here \cite{ksiazka}, at least for those
  who can read it.}.

   \subsection*{Reproducing kernel Hilbert $\alga$--module:
   the properties} The first feature is the reproducing kernel
property \eqref{2.30.1} extends as
    %The defining formula \eqref{1.30.1} can rewritten as
   \begin{equation} \label{2a.30.1}
   F(s)=\is F{K_s}_K,\quad s\in S,\;F\in\eec_K.
   \end{equation}
   What is important when one wants to think of any kind of
minimality is that the inner product $\alga$-module $\ddc_K$ is
already \underbar{dense} in $\eec_K$ by the construction.
 %  \begin{rem} \tlabel{t2.12.2}
%   Let us make the following calculation
%   \begin{align*}
%\big\|\sum\nolimits_kK_{s_k}\algaa_k\big\|_K^2=
%\big\|\iso{\sum\nolimits_iK_{s_i}\algaa_i}
%{\sum\nolimits_jK_{s_j}\algaa_j}_K\big\|_\alga
%=\sum\nolimits_{i,j}\algaa_i^*K(s_i,s_j)\algaa_j,
%   \end{align*}
%
%   \end{rem}

   The {\em evaluation mapping} $\varphi_s$ at $s\in S$ given by
   \begin{equation*}
   \funkc {\varphi_s}{\eec_K}F{F(s)}\alga
   \end{equation*}
   is \alga--linear and, due to (iii) of Proposition
\ref{t1.31.1}, bounded. Moreover,
   \begin{equation*}
\is{\varphi_s(K_t)}\algaa_\alga=K_t(s)^*\algaa=K(t,s)\algaa=
\is{K_t}{K_s\algaa}_K.
   \end{equation*}
   which means $(\varphi_s)^* \algaa=K_s\algaa$. All this leads to
 what some people (including the author) may see as the archetype
 of the Kolmogorov decomposition.
   \begin{thm}\tlabel{t2.31.1}
   For every $s\in S$, $\varphi_s\in{\bf B}^*(\eec_K,\alga)$ and
$(\varphi_s)^*$ acts as $(\varphi_s)^* \algaa=K_s\algaa$.

\noindent Moreover,
   \begin{equation} \label{3.31.1}
   K(s,t)=\varphi_s(\varphi_t)^*,\quad s,t\in S.
   \end{equation}
   \end{thm}
   \begin{rem} \tlabel{t1.2.1}
    Every Hilbert \alga--module $\varXi$ is a reproducing kernel
Hilbert \alga--module over itself with the kernel $K$ defined as
   \begin{equation*}
   K(\xi,\eta)\okr\is \xi\eta_K,\quad \xi,\eta\in\varXi.
   \end{equation*}
  % Suppose we are given two Hilbert \alga--modules $\varXi_1$ and $$
   \end{rem}

   \begin{note}
   The amazing grace of the reproducing kernel Hilbert spaces,
when they are constructed according to the rules contained in
\cite{aro}, is in the space is composed of \underbar{functions}.
The same happens also to the Hilbert C${}^*$--modules. However,
the latter lack the RKHS test unless the module is self--dual.
This is so because the Riesz representation theorem, which is the
only reason for the test to work, fails to hold.
   \end{note}
    There is an occurrence when the RKHS \underline{test} is
certainly valid\,\footnote{\;For the Hilbert space case cf.
\cite{cue}, property ($\eta$).} too. It allows to determine
precisely which {\alga}--functions constitute the space $\eec_K$.
This highlights the extraordinary features of the reproducing
kernel spaces structure.
   \begin{pro}\tlabel{t1.22.2}
   Suppose the {\alga}--Hilbert module $\eec_K$ is selfdual. For
an $\alga$--function $F$ on $S$ the following conditions are
equivalent:
   \begin{enumerate}
   \item[\kolo 1] $F$ belongs to $\eec_K$; \item[\kolo 2] for any
finite choice of $(s_m)_m\subset S$ and $(\algaa_n)_n\subset\alga$
   \begin{equation*}
\sum_{k,l}\algaa^*_kF(s_k)\gw F(s_l)\algaa_l \Le
\sum_{k,l}\algaa^*_kK(s_k,s_l)\algaa_l.
   \end{equation*}
   \end{enumerate}
   \end{pro}
   \begin{proof}
   If uses the same argument as that in \cite{cue}.
   \end{proof}

   \begin{com}
   An important representative of kernels with separated
variables, as opposed to what is going to follow, is that allied
to the cosine function as well as to those alike, see for instance
\cite{niech}.
   \end{com}

    \subsection*{Semigroups in action} Theorem \ref{t2.31.1} is a ground floor version
of Murphy's Theorem 2.3. Specifying $\alga={\bf B}^*(\varXi)$,
where $\varXi$ \underbar{is} already a Hilbert C${}\gw$--module,
we may go upstairs to get precisely that Theorem. However, we
prefer still to keep moving on the ground floor and pass to
Theorem 2.4 of \cite{murphy} this route.

   Let $\sem$ be a multiplicative semigroup of \underbar{left}
actions on $S$. Let us define two operators in $\eec_K$ related to
a given $\sems$. For $F\in\eec_K$ and $\sems\in\sem$ define first
the $\sems$'s translate $F_{[\sems]}$ of $F$ as $F_{[\sems]}\okr
F(\sems s)$, $s\in S$. The operator $\varPsi_K(\sems)$ is well
defined by
   \begin{equation*}
\dz{\varPsi_K(\sems)}\okr\zb{F\in\eec_K}{F_{[\sems]}\in\eec_K},\quad
\varPsi_K(\sems)F\okr F_{[\sems]}.
   \end{equation*}
   $\varPsi_K(\sems)$ may be an unbounded operator with domain
$\dz{\varPsi_K(\sems)}$ different from $\eec_K$. The other
operator, $\varPhi_K(\sems)$ may not be well defined, if it is it
is always densely defined
   \begin{equation}\label{3.16.2}
   \dz{\varPhi_K(\sems)}\okr\ddc_K,\quad
   \varPhi_K(\sems)\sum\nolimits_iK_{ s_i}\algaa_i=
   \sum\nolimits_iK_{\sems s_i}\algaa_i,\quad (s_i)_i\subset S.
   \end{equation}

   The reproducing kernel property \eqref{2.30.1} implies
   \begin{equation*} %\label{14.11.1}
   \is{F}{K_{\sems s}}_K=\is{\varPsi_K(\sems)F}{K_{s}}_K,\quad
   F\in\ddc(\varPsi_K(\sems)),\; s\in S
   \end{equation*}
   and this in turn helps to prove the following
   \begin{pro}\tlabel{t1.1.2}
   $\varPsi_K(\sems)$ is a closed operator. If $\varPsi_K(\sems)$
is densely defined, then $\varPhi_K(\sems)$ is well defined and
$\varPhi_K(\sems)\gw=\varPsi_K(\sems)$, and {\em vice versa}.
   \end{pro}

   In principle, $\varPsi(\sems)$ may not be densely defined while
$\varPhi(\sems)$ may not be well defined as an operator. In this
paper we are interested exclusively in the case when these two
operators are bounded. This is case described as follows.
   \begin{pro}\tlabel{t14.11.2}
   $\varPhi_K(\sems)$ is a well defined bounded operator if and
   only if there is a \underbar{number} $c(\sems)\Ge0$ such that
   \begin{equation} \label{14.11.3}
\big\|\sum\nolimits_{i,j}\algaa_i^*K(\sems s_i,\sems
s_j)\algaa_j\big\|_\alga\Le c(\sems)
\big\|\sum\nolimits_{i,j}\algaa_i^*K(s_i,s_j)\algaa_j\big\|_\alga,\quad
(s_m)_m\subset S,\;(\algaa_m)_m\subset\alga .
   \end{equation}
   If this happens then $\varPsi_K(\sems)$ is a densely defined
   bounded operator and
   $\|\varPsi_K(\sems)\|=\|\varPhi_K(\sems)\|\Le c(\sems)$;
   keeping the same notation for the extensions of these operators
   we have
   \begin{equation*}
   \text{$\varPhi_K(\sems)=\varPsi_K(\sems)^*$ and
$\varPsi_K(\sems)$ as well as $\varPhi_K(\sems)$ are in ${\bf
B}^*(\eec_K)$.}
   \end{equation*}
   Moreover, the mapping
   \begin{equation} \label{3.2.2}
   \funkc {\varPhi_K}\sem\sems{\varPhi_K(\sems)}{{\bf
B}^*(\eec_K)}
   \end{equation}
   is multiplicative while the mapping
   \begin{equation} \label{1.13.2}
   \funkc {\varPsi_K}\sem\sems{\varPsi_K(\sems)}{{\bf
B}^*(\eec_K)}
   \end{equation}
   is antimultiplicative.

   \noindent If $\sem$ is unital with unit ${\sf 1}$,
$\varPhi({\sf 1})=\varPsi({\sf 1})={\sf 1}_{\eec_K}=\text{the
identity operator in $\eec_K$}$.
   \end{pro}
   \begin{com}
   The important notice is one of the responsibilities of the
boundedness condition \eqref{14.11.3} is to ensure the operator
$\varPhi(\sems)$ to be \underbar{well} \underbar{defined}; it goes
together unnoticeably with boundedness of $\varPhi(\sems)$. These
two matters make up a juncture.
   \end{com}

    Now the Kolmogorov factorization \eqref{3.31.1} resembles more
than some people still would like to have.
   \begin{cor} \tlabel{t2.2.1}
   Suppose \eqref{14.11.3} holds. Then for $\sems,\semt\in\sem$,
$s,t\in S$
   \begin{gather}\label{1.14.2}
   K(\sems \,s,\semt\, t)=\is{\varPhi_K(\sems)K_s}
{\varPhi_K(\semt)K_t}_K,
   \\\notag K(s,\semt\, t)=\is{K_s}
{\varPhi_K(\semt)K_t}_K .
   \end{gather}
   \end{cor}
 %  Notice that \eqref{1x.14.2} does not follow directly from
%\eqref{1.14.2} as $\sem$ may not be unital.

    The appearance of the homomorphism $\varPhi_K$ makes the
$\alga$--module a \underbar{kind} of {\em C${}^*$--correspondence}
in the sense of \cite{muhly}. We come to this notion closer as we
specify more $S$ and $\sem$.
    \subsection*{Involution in {\sem} added} Suppose $\sem$ has an involution and the action
of {\sem} is {\em transitive} with respect to the kernel $K$ or
the kernel $K$ is {\sem}--{\em invariant}, which means anyway that
   \begin{equation*} %\label{1.2.1}
   K(s,\sems\, t)=K(\sems^*s,t ),\quad \sems\in\sem,\;s,t\in S.
   \end{equation*}
  % Denote by $\sem_K$ the multiplicative $*$--semigroup generated
%in ${\bf B}^*(\eec_K)$ by all $\varPsi_K(\sems)$ and
%$\varPsi_K(\semt)^*$, $\sems,\semt\in\sem$.
   \begin{cor} \tlabel{t2.2.2}
    Suppose \eqref{14.11.3} holds. Then
$\varPsi_K(s^*)=\varPhi_K(s)$ and the homomorphism $\varPhi$ as in
\eqref{3.2.2} becomes a $*$--homomorphism while the other
$\varPsi$, that in \eqref{1.13.2}, is anti--$*$--homomorphism.
%Therefore, $\sem_K$ can be thought of as a $*$--semigroup acting
%from the \underbar{left} on ${\bf B}^*(\eec_K)$.
  % \begin{equation} \label{2.14.2}
%   \omega(\sems^* \,s^\semt\, t)=\is{\varPhi_K(\sems)K_s}
%{\varPhi_K(\semt)K_t}_K, \quad \sems,\semt\in\sem,\;s,t\in S.
%   \end{equation}

   \end{cor}

    Now we have to change the \underbar{meaning}
   \begin{center}
   $S=\sem$ a $*$--semigroup, \\
$K(\sems,\semt)=\omega(\sems^*\semt)$ with $\funk\omega\sem\alga$
   \end{center}
   and the \underbar{notation}
   \begin{center}
   $\ddc_\omega$, $\eec_\omega$, $\omega_\sems$, $\isp_\omega$,
$\varPhi_\omega$ and $\varPsi_\omega$\\ instead
of\\
$\ddc_K$, $\eec_K$, $K_\sems$, $\isp_K$, $\varPhi_K$ and
$\varPsi_K$ .
   \end{center}
   According to \eqref{1.25.1}, $\alga$--positive definiteness of
$\omega$ means,
    \begin{equation*} %\label{1a.25.1}
\text{$\sum_{k,l}\algaa^*_k\omega(\sems_k^*\sems_l)\algaa_l\Ge0$
for any finite choice of $(\sems_n)_n\subset \sem$ and
$(\algaa_n)_n\subset\alga$}
   \end{equation*}
   and the decomposition \eqref{1.14.2} (and together with the
defining formula \eqref{3.16.2}) takes the form
   \begin{equation*}%\label{1a.14.2}
\omega(\sems_1^*\sems_2^*\semt_2\semt_1)=
\is{\varPhi_\omega(\sems_2)\omega_{\sems_1}}
{\varPhi_\omega(\semt_2)\omega_{{\semt}_1}}_\omega=
\is{\omega_{\sems_2\sems_1}}{\omega_{\semt_2\semt_1}}_\omega,
\quad \sems_1,\sems_2,\semt_1,\semt_2\in\sem.
   \end{equation*}
   Via the reproducing kernel Hilbert $\alga$-module construction,
the inequalities (ii) and (iii) of Proposition
   \ref{t1.31.1}, which we are going to use, now take the form
   \begin{multline}  \label{3.10.2}
    \big(\sum\nolimits_{i,k}\algaa_i^*\omega
    (\sems_i^*\semt_k)\algbb_k\big)^*
\sum\nolimits_{i,k}\algaa_i^*\omega
    (\sems_i^*\semt_k)\algbb_k \\\Le
    \big\|\sum\nolimits_{i,j}\algaa_i^*\omega
    (\sems_i^*\sems_j)\algaa_j\big\|_\alga
    \,\sum\nolimits_{k,l}\algbb_k^*\omega
    (\semt_k^*\semt_l)\algbb_l,
   \\(\sems_m)_m,(\semt_n)_n\subset \sem ,\;(\algaa_m)_m,
(\algbb_n)_n \subset\alga
   \end{multline}
   and
    \begin{multline} \label{1.10.2}
    \big\|\sum\nolimits_{i,k}\algaa_i^*\omega
    (\sems_i^*\semt_k)\algbb_k\big\|_\alga^2 \Le
    \big\|\sum\nolimits_{i,j}\algaa_i^*\omega
    (\sems_i^*\sems_j)\algaa_j\big\|_\alga\,
    \big\|\sum\nolimits_{k,l}\algbb_k^*\omega
    (\semt_k^*\semt_l)\algbb_l\big\|_\alga,
   \\(\sems_m)_m,(\semt_n)_n\subset \sem ,\;(\algaa_m)_m,
(\algbb_n)_n \subset\alga.
   \end{multline}
    Notice that
   \begin{equation} \label{1.11.2}
   \big\|\sum\nolimits_{i,j}\algaa_i^*\omega
(\sems_i^*\sems_j)\algaa_j\big\|_\alga=\big\|\iso{\sum\nolimits_{i}
\omega_{\sems_i}\algaa_i}{\sum\nolimits_{i}
\omega_{\sems_i}\algaa_i}\big\|_\alga=\big\|\sum\nolimits_{i}
\omega_{\sems_i}\algaa_i\big\|_{\omega}^2.
   \end{equation}

   \begin{rem} \tlabel{t2.22.2}
   If $\omega\in\eec_\omega$ then
$\varPhi(\sems)\omega=\omega_\sems$ for all $\sems$ regardless
$\sem$ or $\alga$ is unital or not. Indeed, for any $\semt\in\sem$
   \begin{equation*}
\is{\varPsi(\sems)\omega}{\omega_\semt}_\omega=\is{\omega_{[s]}}
{\omega_\semt}=\omega(\sems\semt)=\omega_{\sems^*}(\semt)=
\is{\omega_{\sems^*}}{\omega_\semt}.
   \end{equation*}
   Because $\varPhi(\sems)=\varPsi(\sems^*)$ and $\omega_\semt$'s
span $\eec_\omega$, we get it.
   \end{rem}

    \section*{Towards the KSGNS theorem}
     \subsection*{The boundedness condition, several versions} We turn to the {\em boundedness condition}
     \eqref{14.11.3} which under the current circumstances takes
     the form (a) below.
   \begin{pro}\tlabel{t1.10.1}
   The following conditions are equivalent:
   \begin{enumerate}
   \item[(a)] for every $\sems\in\sem$ there is a constant
   $c(\sems)\Ge0$ such that
   \begin{equation*}
\big\|\sum\nolimits_{i,j}\algaa_i^*\omega
(\sems_i^*\sems^*\sems\sems_j)\algaa_j\big\|_\alga\Le
c(\sems)\big\|
\sum\nolimits_{i,j}\algaa_i^*\omega(\sems_i^*\sems_j)\algaa_j\big\|_\alga
,\quad (\sems_m)_m\subset \sem,\;(\algaa_m)_m\subset\alga;
   \end{equation*}
   \item[(b)] for every $\sems\in\sem$ there is a constant
   $c(\sems)\Ge0$ such that
   \begin{equation*}
\big\|\algaa^*\omega (\semt^*\sems^*\sems\semt)\algaa\big\|_\alga
\Le c(\sems)\big\| \algaa^*\omega(\semt^*\semt)\algaa\big\|_\alga
,\quad \semt\in \sem,\;\algaa\in\alga;
   \end{equation*}
   \item[(c)] there is a submultiplicative function $\funk
   c\sem{[0,+\infty)}$ such that
   \begin{equation*}
   \big\|\algaa^*\omega
   (\semt^*\sems^*\sems\semt)\algaa\big\|_\alga\Le
   d(\semt,\algaa)c(\sems),\quad \sems,\semt\in
   \sem,\;\algaa\in\alga;
   \end{equation*}
   \item[(d)] $\liminf_n\big\|\sum\nolimits_{i,j}\algaa_i^*\omega
   (\sems_i^*(\sems^*\sems)^{2^n}\sems_j)\algaa_j\big\|_\alga^{2^{-n}}$
   is finite and does not depend on the choice of $(\sems_n)_m$
   and $(\algaa_m)_m$.
   \end{enumerate}
   \end{pro}
   \begin{proof}
   (a) \wynik (b) trivially. To get (c) from (b) notice that
   requiring $c(\sems)$ to be minimal in (b) uniformly in $\semt$
   and $\algaa$ implies submultiplicativity of such a function
   $\funk c\sem{[0,+\infty)}$. Applying \eqref{1.10.2} we get
   \begin{equation*}
   \big\|\algaa_i^*\omega
   (\sems_i^*\sems^*\sems\sems_j)\algaa_j\big\|_\alga^2 \Le
   \big\|\algaa_i^*\omega
   (\sems_i^*\sems^*\sems\sems_i)\algaa_i\big\|_\alga
   \big\|\algaa_j^*\omega
   (\sems_j^*\sems^*\sems\sems_j)\algaa_j\big\|_\alga.
   \end{equation*}
   Therefore,
   \begin{multline*}
  \big\|\sum_{i,j}\algaa_i^*\omega
  (\sems_i^*(\sems^*\sems)^{2^n}\sems_j)\algaa_j\big\|_\alga^{2^{-n}}
  \Le \big(\sum_{i,j}\big\|\algaa_i^*\omega
  (\sems_i^*(\sems^*\sems)^{2^n}\sems_i)\algaa_i\big\|_\alga^{1/2}
  \big\|\algaa_j^*\omega
  (\sems_j^*(\sems^*\sems)^{2^n}\sems_j)\algaa_j\big\|_\alga^{1/2}\big)^{2^{-n}}
   \\\Le d
\big(c((\sems^*\sems)^{2^{n-1}})^{1/2}
c((\sems^*\sems)^{2^{n-1}})^{1/2}\big)^{2^{-n}}\Le d
c(\sems^*\sems)^{-1/2},
   \end{multline*}
   with $d\okr\max\nolimits_i\{d(\sems_i,\algaa_i)\}$. Thus (d)
   follows.

   Now repeated use of \eqref{1.10.2} with $\algbb_i=\algaa_i$ and
   $\semt_i=\sems^*\sems\sems_i$ gives us
   \begin{equation*}
    \big\|\sum\nolimits_{i,j}\algaa_i^*\omega
    (\sems_i^*\sems^*\sems\sems_j)\algaa_j\big\|_\alga^2 \Le
    \big\|\sum\nolimits_{i,j}\algaa_i^*\omega
    (\sems_i^*(\sems^*\sems)^{2^k}\sems_j)\algaa_j\big\|_\alga^{2^{-k}}
    \big\|\sum\nolimits_{k,l}\algaa_k^*\omega
    (\sems_k^*\sems_l)\algaa_l\big\|_\alga^{1-2^{-k}}.
   \end{equation*}
The limit passage, after taking into account (d), leads to (a).
   \end{proof}
   Condition (b) may be viewed as a \underbar{diagonalization} of
(a).
   \begin{rem} \tlabel{t1.4.3}
   The above versions of the boundedness condition have been
discussed by the present author on different occasions, always for
positive definite operator valued kernels. Condition (a) is
Sz.--Nagy's boundedness condition in his general dilation theorem
in \cite{app}. Condition (b) is in \cite{ark}, condition (c) is
singled out in \cite{sz_pams} and condition (d), the forerunner of
the whole case, is already in \cite{sz_bull}.
   \end{rem}

   \begin{com}\label{c1}
   Notice either $\sem$ or $\alga$ need not be unital for the
conclusion of Proposition \ref{t1.10.1} to hold. Anyway, now any
of (a)--(d) may be viewed as a {boundedness condition} for
$\omega$. Moreover, the conditions (c) extends from $\sem$ to
$\sem^{+}$ with ease.
   \end{com}

   \subsection*{Unitization of $\sem$} If $\sem$ has no unit define its {\em unitization}
$\sem^{+}$ as ${\sem^{+}}\okr {\sem}\cup\{\sf 1\}$ with semigroup
multiplication and involution as ${\sf 1}\sems=\sems{\sf 1}=\sems$
and ${\sf 1}^*={\sf 1}$. If $\sem$ is already unital, that is it
has a unit, for homogenization purpose set $\sem^{+}\okr\sem$.
   The following is important enough to be particularized.
\vspace{1pt}\begin{equation} \boxed{ \text{$\omega$ belongs to
$\eec_\omega$ and $\omega(\sems^*)=\omega(\sems)^*$,
$\sems\in\sem$.} } \tag{$\boldsymbol{\ast}$}
   \end{equation}

   \begin{pro}\tlabel{t2.8.2}
   Suppose the C${}\gw$--algebra $\alga$ is {\em unital} with the
unit denoted by ${\sf e}$. Consider the following conditions:
   \begin{enumerate}
   \item[($\alpha$)] with some $c>0$
   \begin{gather*}
   \omega(\sems^*)=\omega(\sems)^*,\quad \sems\in\sem, \\
c\sum\nolimits_{i,,j}\algaa_i^*\omega(\sems_i)^*
\omega(\sems_j)\algaa_j \Le \sum\nolimits_{i,j}\algaa_i^*
\omega(\sems_i^*\sems_j)\algaa_j,\quad (\sems_m)_m\subset \sem
,\;(\algaa_m)_m\subset\alga;
   \end{gather*}

    \item[($\beta$)] $\funk\omega{\sem}{\alga}$ extends to a
positive definite function $\funk{\omega^{+}}{\sem^{+}}{\alga}$;
   \item[($\gamma$)] condition {\rm(}{$\boldsymbol{\ast}$}{\rm)}
holds.
   \end{enumerate}
   Then {\rm(}$\gamma${\rm)} \wynik {\rm(}$\alpha${\rm)}
$\Longleftrightarrow$ {\rm(}$\beta${\rm)}.
 If $\eec_\omega$ is selfdual, then {\rm(}$\alpha${\rm)} \wynik
{\rm(}$\gamma${\rm)}.
   \end{pro}
   An $\alga$-function $\omega$ on $\sem$ satisfying any of the
equivalent conditions ($\alpha$) or ($\beta$) of Proposition
\ref{t2.8.2} is said to have the {\em extension property}. If
$\sem$ is unital this is nothing but $\alga$--positive
definiteness (actually, the second condition in ($\alpha$)
guarantees this at once).

The equivalence {\rm(}$\alpha${\rm)} $\Longleftrightarrow$
{\rm(}$\beta${\rm)} for Hilbert space operator valued functions is
in \cite{sz_ann}, for Hilbert C${}^*$--module valued ones in
\cite{itoh2}.
   \begin{proof}
   Suppose ($\alpha$) holds. Extending $\omega$ to $\omega^{+}$ by
$\omega^{+}({\sf 1})\okr c^{-1/2}{{\sf e}}$ we have
   \begin{align*}
   \sum\nolimits_{i,j}\algaa_i^*\omega^{+}
({\sems_i}^*{\sems_j})\algaa_j
   &=\sum\nolimits_{{\sems_i}\neq{{\sf
1}}\neq{\sems_j}}\algaa_i^*\omega ({\sems_i}^*{\sems_j})\algaa_j
+\sum\nolimits_{{\sems_i}\neq{{\sf 1}}}\algaa_i^*
\omega({\sems_i}^*)\algaa_{\sf 1}
   \\&+
\sum\nolimits_{{\sems_j}\neq{{\sf 1}}}\algaa_{\sf 1}^*\omega
{(\sems_j)}\algaa_j + \algaa_{\sf 1}^*c^{-1}\algaa_{\sf 1}
  % \\&\Ge c\big(\|\sum\nolimits_{{\sems_i}\neq{{\sf 1}}}
%\omega({\sems_i})\algaa_i\|_{\alga}^2{\sf e}
%+\sum\nolimits_{{\sems_i}\neq{{\sf 1}}}c^{-1}\algaa_i^*
%\omega({\sems_i})^*\algaa_{\sf 1}
%   \\&+\sum\nolimits_{{\sems_j}\neq{\sf 1}}c^{-1}\algaa_
%{\sf 1}^*\omega {(\sems_j)}\algaa_j +c^{-2} \algaa_{{\sf
%1}}^*\algaa_{\sf 1}\big)
   \\&
{\Ge} c\big(\sum\nolimits_{{\sems_i}\neq{{\sf 1}}\neq{{\sems_j}}}
\algaa_i^*\omega({\sems_i})^*\omega({\sems_j})\algaa_j
+\sum\nolimits_{{\sems_i}\neq{{\sf 1}}}c^{-1}\algaa_i^*
\omega({\sems_i})^*\algaa_{\sf 1}
   \\&+\sum\nolimits_{{\sems_j}\neq{\sf 1}}c^{-1}\algaa_
{\sf 1}^*\omega {(\sems_j)}\algaa_j +c^{-2} \algaa_{{\sf
1}}^*\algaa_{\sf 1}\big)
   \\&=  c\big(\sum\nolimits_{{\sems_i}\neq{{\sf 1}}}
\omega({\sems_i})\algaa_{\sems} +c^{-1}\algaa_{\sf
1}\big)^*\big(\sum\nolimits_{{\sems_i}\neq{{\sf 1}}}
\omega({\sems_i})\algaa_{\sems} +c^{-1}\algaa_{\sf 1}\big) \Ge0.
   \end{align*}
Thus $\omega^{+}$ is positive definite on $\sem^{+}$.

   Back to the proof suppose $\omega$ is extendible, that is
$(\beta)$ holds. Writing \eqref{3.10.2} for $\omega^+$ with
$\sems_i={\sf 1}$ and $\algaa_k={\sf e}$ and then restricting the
resulting inequality to $\sem$ (remember,
$\omega^+(\sems)=\omega(\sems)$ for $\sems\in\sem$) we get the
second of ($\alpha$). The first comes from \eqref{5.10.2}.

%Because $\omega^+\in\eec_{\omega^+}$ and
%$\is{\omega^+_\sems}{\omega^+}_{\omega^+}=\omega^+(\sems)=\omega(\sems)$
%for $\sems\in\sem$, Annex allows as to place $\omega$ in
%$\eec_\omega$. The other part of ($\gamma$) comes from that for
%$\omega^+$.
%
 Suppose now ($\gamma$) holds. Then the reproducing kernel
property
   \begin{equation*}
   \omega(\sems)=\is \omega{\omega_\sems}_\omega
   \end{equation*}
   when combined with condition (ii) of Proposition \ref{t1.31.1}
gives
   \begin{align} \label{1.17.2}
   \begin{split}
   \sum\nolimits_{i,,j}\algaa_i^*\omega(\sems_i)^*
\omega(\sems_i)\algaa_i &= \iso \omega
{\sum\nolimits_{i}\omega_\sems\algaa_i}_\omega^*\iso
\omega{\sum\nolimits_{j}\omega_\sems\algaa_j}_\omega
   \\&\Le
\|\omega\|_{\omega}^2\iso{\sum\nolimits_{i}\omega_\sems\algaa_i}
{\sum\nolimits_{j}\omega_\sems\algaa_j}_\omega \\&\Le
\|\omega\|_{\omega}^2 \sum\nolimits_{i,j}\algaa_i^*
\omega(\sems_i^*\sems_j)\algaa_j.
   \end{split}
   \end{align}
   This is the second of $(\alpha)$, the first has just to be
copied.

   Suppose now the Hilbert {\alga}--module $\eec_\omega$ is
selfdual. Then the inequality in $(\alpha)$ fits in the condition
\kolo 2 of Proposition \ref{t1.22.2} with $F=\omega$ and the last
conclusion follows
   \end{proof}
   Call a net $({\sf 1}_\lambda)_\lambda\subset\sem$ an {\em
approximate unit} for $\omega$ if there is a net
$(\algaa_\lambda)_\lambda\subset\alga$ such that
   \begin{gather} \label{5.22.2}
   \text{$\algaa_\lambda^*\omega({\sf 1}_\lambda\sems)
   {\stackrel{{\scriptscriptstyle{\mathsf{\alga}}}}{\longrightarrow}}
   \omega(\sems)$ and $\omega(\sems{\sf
   1}_\lambda^*)\algaa_\lambda
   {\stackrel{{\scriptscriptstyle{\mathsf{\alga}}}}{\longrightarrow}}
   \omega(\sems)$},
   \\\label{5a.22.2}
   \text{$\|\algaa_\lambda^*\omega({\sf 1}_\lambda{\sf
   1}_\lambda^*)\algaa_\lambda\|_\alga$ is bounded in $\lambda$};
   \end{gather}
   it is called a {\em strong approximate unit} for $\omega$ if
\eqref{5.22.2} holds and, instead of \eqref{5a.22.2},
   \begin{equation} \label{5b.22.2}
   \text{$(\algaa_\lambda^*\omega({\sf 1}_\lambda{\sf
1}_\lambda^*)\algaa_\lambda)_\lambda$ is a Cauchy net}.
   \end{equation}
   Notice that \eqref{5b.22.2} implies \eqref{5a.22.2}, thus
strong approximate unit seems to be really \underbar{stronger}. Is
it?
   \begin{pro}\tlabel{t3.22.2}
  Suppose $\alga$ is unital. If $\omega$ has an approximate unit
$({\sf 1}_\lambda)_\lambda$ then it has an extension property. If
$\omega$ has a strong approximate unit $({\sf 1}_\lambda)_\lambda$
then $(\boldsymbol{\ast})$ holds. Conversely, if
$(\boldsymbol{\ast})$ holds then there are two arrays
$((\sems^n_{i})_{i\in\{{\rm finite\}}})_{n=0}^\infty \subset\sem$
and $((\algaa^n_{i})_{i\in\{{\rm finite\}}})_{n=0}^\infty
\subset\alga$ such that
    \begin{gather} \label{5x.22.2}
   \text{$\sum\nolimits_i(\algaa_i^n)^*\omega({\sems}_i^n\sems)
   {\stackrel{{\scriptscriptstyle{\mathsf{\alga}}}}{\longrightarrow}}
   \omega(\sems)$ and
   $\sum\nolimits_i\omega(\sems({\sems}_i^n)^*)\algaa_i^n
   {\stackrel{{\scriptscriptstyle{\mathsf{\alga}}}}{\longrightarrow}}
   \omega(\sems)$},
   \\\label{5ax.22.2}
\text{$\sum\nolimits_{i,j}(\algaa_i^n)^*\omega
({\sems}_i^n({\sems}_j^n)^*) \algaa_j^n$ is a Cauchy sequence in
$n$};
   \end{gather}
   \end{pro}
   \begin{proof}
   Insert the approximate unit into \eqref{5.10.2} and
\eqref{3.10.2} and proceed in an appropriate way so as get both
parts of ($\alpha$).

   Suppose $\omega$ has a strong approximate unit $({\sf
1}_\lambda)_\lambda$. Then, \eqref{5b.22.2} with a little help of
 \eqref{1.11.2} implies $(\omega_{{\sf
1}_\lambda^*}\algaa_\lambda)_\lambda$ is a Cauchy net in
$\eec_\omega$. Therefore, there is $F\in\eec_\omega$ such that
$\omega_{{\sf 1}_\lambda^*}\algaa_\lambda
{\stackrel{{\scriptscriptstyle{\mathsf{\eec_\omega}}}}{\longrightarrow}}
F$ and, consequently,
   \begin{equation*}
\omega(\sems)
{\stackrel{{\scriptscriptstyle{\mathsf{\alga}}}}{\longleftarrow}}
 \omega({\sf 1}_\lambda \sems)\algaa_\lambda=\is{\omega_{{\sf
1}_\lambda^*}\algaa_\lambda} {\omega_\sems}
{\stackrel{{\scriptscriptstyle{\mathsf{\alga}}}}{\longrightarrow}}
\is F{\omega_\sems} =F(\sems),\quad \sems\in\sem.
   \end{equation*}
  Therefore, $\omega=F\in\eec_\omega$. The Hermitian symmetry of
$\omega$ goes straightforwardly from \eqref{5.10.2}.

   If $\sem$ is unital then \eqref{5x.22.2} and \eqref{5ax.22.2}
hold with the trivial choice of the required data. If not, then
there is always a sequence
$(\sum\nolimits_i\omega_{({\sems}_i^n)^*}\algaa_i^n)_n $ converges
to $\omega$ in $\eec_\omega$ (density of $\ddc_\omega$\,!) and so
does $\varPhi_\omega
(\sems)(\sum\nolimits_i\omega_{({\sems}_i^n)^*}\algaa_i^n)_n $ to
$\omega(\sems)$, use Remark \ref{t2.22.2} on a way. This gives
\eqref{5ax.22.2}. Because, due to the reproducing property norm
convergence implies pointwise, \eqref{5x.22.2} follows as well.
   \end{proof}

    \begin{rem}\label{t5.22.2}
   If $\omega$ satisfies the extension property, then the mapping
   $\funk W {\omega_\sems} {\omega_\sems^+}$, $\sems\in\sem$,
   extends to a linear operator of $\eec_\omega$ into
   $\eec_{\omega^+}$. The operator $W$ is adjointable with adjoint
   given by $W\gw \omega_\semt^+=\omega_\semt$ if $t\neq{\sf 1}$
   and $0$ otherwise.
    Furthermore, $W$ is an $\alga$ isometry onto
   $\eec_{\omega^+}^0\okr\clolin\zb{\omega_\sems^+\algaa}{\sems\in\sem
   ,\,\algaa\in\alga}$. It is a matter of direct verification that
   the basic RKHS operators are related by
   \begin{equation*} %\label{2.16.2}
W\varPhi_\omega(\sems)W\gw =\varPhi_{\omega^+}(\sems),\quad
\sems\in\sem.
   \end{equation*}
   \end{rem} \vspace{2pt}
   \subsection*{The basic dilation theorem} Given an $\alga$--function $\omega$ on a
$*$--semigroup $\sem$, it is clear that if there exists a Hilbert
$\alga$--module $\eec$, a multiplicative $*$--homomorphism
$\funk\varPhi \sem{{\bf B}^*(\eec)}$ and $V\in{\bf
B}^*(\alga,\eec)$ such that
   \begin{equation}\label{1.23.2}
   \omega(\sems)=V\gw\varPhi(\sems)V,\quad \sems\in\sem,
   \end{equation}
   then $\omega$ is {\alga}--positive definite. Moreover,
\eqref{1.23.2} implies
   \begin{equation}\label{6.23.2}
   \big\|\algaa^*\omega
   (\semt^*\sems^*\sems\semt)\algaa\big\|_\alga=
   \|\varPhi_\omega(\sems)V\algaa\|_\eec
   \Le\|V\algaa\|_\eec\,\|\varPhi(\sems)\|_\eec
   \end{equation}
   with $\|\varPhi_\omega(\sems)\|_\omega $ submultiplicative.
Thus (c) of Proposition \ref{t1.10.1} holds. To show $\omega$ has
an extension property use \eqref{1.23.2} like in \eqref{1.17.2};
the Hermitian symmetry of $\omega$ follows from \eqref{1.23.2}
immediately. This solves in the rather trivial way most of the one
side of the dilation story. The other is the masterpiece for those
who may appreciate it. We are aware of the fact that the category
theory fans are going to be disappointed; RKHS is too reach in
information it carries to be a categorical object on call.

   Another issue which we are \underbar{not} going to touch is
\underbar{uniqueness} of minimal dilations whatever the latter
means. This can be done in a standard way anytime a need appears.
The notion of minimality in \eqref{6.23.2} has to be introduced
anyway. We say that the triple $(\eec,\varPhi,V)$ are {\em
minimal}\,\footnote{\;Other names which appear on this occasion
are {\em nondegenerate} or {\em essential}.} for $\omega$ if
$\eec=\clolin\zb{\varPhi(\sems)V\algaa}
{\sems\in\sem\;\algaa\in\alga}$. Minimality is always done when
$\sem$ is unital.

   \begin{thm}\label{t1.24.2}
    Let $\alga$ be a unital $C^*$--algebra. For an
$\alga$--positive definite function $\omega$ on {\sem} the
following conclusions hold.
   \begin{enumerate}
   \item[\liczp 1] Then there is $V^+\in{\bf
B}^*(\alga,\eec_{\omega^+})$ such that
   \begin{equation} \label{1.8.2}
   \algaa^*\omega(\sems)\algbb= \is{V^+\algaa}
{\varPhi_\omega^+(\sems)V^+\algbb}_{\omega^+}, \quad
\sems\in\sem,\;\,\algaa,\algbb\in\alga
   \end{equation}
   if and only if $\omega$ satisfies the boundedness condition,
that is any of the conditions {\rm (a)--(d)} of {\rm Proposition
\ref{t1.10.1}} as well as it has the extension property, that is
any of the conditions $(\alpha)$--$(\beta)$ of {\rm Proposition
\ref{t2.8.2}} holds.
   \item[\liczp 2] Suppose condition {$(\boldsymbol{\ast})$}
holds\,\footnote{\;This happens when $\omega$ has a strong
approximate unit, cf.Remark \ref{t2.22.2}}. Then there is
$V\in{\bf B}^*(\alga,\eec_{\omega})$ such that
   \begin{equation} \label{1x.8.2}
   \algaa^*\omega(\sems)\algbb= \is{V\algaa}
{\varPhi_\omega(\sems)V\algbb}_{\omega}, \quad
\sems\in\sem\;\,\algaa,\algbb\in\alga.
   \end{equation}
   The operator $V$ is defined by \eqref{5.14.2} and its adjoint
by \eqref{2.23.2}. Therefore, \eqref{1x.8.2} can be written as
   \begin{equation} \label{1y.8.2}
   \algaa^*\omega(\sems)\algbb= \is{\omega\algaa}
{\varPhi_\omega(\sems)\omega\algbb}_{\omega}, \quad
\sems\in\sem\;\,\algaa,\algbb\in\alga.
   \end{equation}

\item[\liczp 3] The triple $(\eec_\omega,\varPhi_\omega,V)$
generated in \liczp 2 is \underbar{minimal}. Consequently,
\eqref{1x.8.2} can be written as
   \begin{equation} \label{1.3.3}
   \omega(\sems)=V\gw \varPhi(\sems)V,\quad \sems\in\sem.
   \end{equation}

   \end{enumerate}
   \end{thm}
  Notice that in the case $\sem$ is unital the extendibility
procedure is needless, $\alga$--positive definiteness and the
boundedness condition are enough to play the game. The whole
embarrassment is caused by a possible lack of unit in $\sem$; this
happens in C${}^*$--algebras but that case is always weaponed with
a substitute, approximate unit. On the other hand, the reproducing
kernel construction we have carried out here in full allows to
raise condition ({$\boldsymbol{\ast}$}). Its simplicity is an
effective counterbalance for more sophisticated technology. At
least it reduces the number of new objects involved from two to
one: just the $\varPhi_\omega$. If {\alga} is unital, the dilation
formula \eqref{1y.8.2} takes the shortest possible form
   \begin{equation} \label{1z.8.2}
  \omega(\sems)= \is{\omega}
{\varPhi_\omega(\sems)\omega}_{\omega}, \quad \sems\in\sem.
   \end{equation}
   \begin{proof}
    The kernel ${(\sems,\semt)}\to{\omega(\sems^*\semt)}$ is
    positive definite in the sense of \eqref{1.25.1} therefore the
    outcomes of the reproducing kernel construction are at our
    disposal.

 Let us prove first conclusion \liczp 2. Due to
({$\boldsymbol{\ast}$}), $\omega\in\eec_\omega$ and the
reproducing kernel property \eqref{2a.30.1} gives us
   \begin{equation*}
   \omega(\sems)=\is\omega{\omega_\sems}_\omega,\quad \sems\in\sem
\end{equation*}
   and this, in turn, allows us to write
    \begin{equation}\label{4.14.2}
  \is{\omega(\sems)
\algaa}\algbb_\alga=\algaa^*\omega(\sems)^*\algbb=\algaa^*\is{\omega_\sems}
\omega_{\omega}\algbb=\is{\omega_\sems\algaa}{\omega\algbb}_\omega.
   \end{equation}
   Reading \eqref{4.14.2} in a proper way (remember $\alga$ is
unital) we infer that $\omega(\sems) \algaa$ is the only ${\sf c}$
such that $\is{\sf
c}\algbb_\alga=\is{\omega_\sems\algaa}{\omega\algbb}_\omega$. This
means that
   \begin{equation} \label{5.14.2}
   \funkc V\alga\algaa{\omega\algaa}{\eec_\omega}
   \end{equation}
   is adjointable with
    \begin{equation} \label{2.23.2}
   \funkc {V\gw} {\eec_\omega}
{\omega_\sems\algaa}{\omega(\sems)\algaa } \alga.
   \end{equation}
   By Remark \ref{t2.22.2}, we have
   \begin{align*}
\is{V\algaa}{\varPhi_\omega(\sems)V\algbb}_\omega=\is
{\omega\algaa}{\varPhi_\omega(\sems)\omega\algbb}_\omega = \is{
\omega\algaa}{
\omega_\sems\algbb}_{\omega}=\algaa^*\omega(\sems)\algbb
   \end{align*}
  and this establishes \eqref{1.8.2}.

    The essential direction in \liczp 1 goes as follows (the
reverse has been already discussed before the theorem). If $\sem$
is unital, we are in a position of \liczp 2. If not,
 we continue the play between $\sem$ and $\sem^+$. Keeping the
notation up for the operator $V$ in the unital case of $\omega^+$,
as done in \liczp 2 for this case, we have it in ${\bf
B}^*(\alga,\eec_{\omega^+})$. Setting now $V^+\okr VW$ we use the
merits of Remark \ref{t5.22.2} to come to \eqref{1.8.2}.

   Putting to use again Remark \ref{t2.22.2} we get
   \begin{equation*}
\sum\nolimits_i\varPhi_{\omega}(\sems_i)Va_i=
\sum\nolimits_i\varPhi_{\omega}(\sems_i)\omega
a_i=\sum\nolimits_i\omega_{\sems_i}a_i
   \end{equation*}
    which makes all the arrangements for minimality of the triple
$(\eec_\omega,\varPsi_\omega,V)$ as
$\sum\nolimits_i\omega_{\sems_i}a_i $ span linearly $\ddc_\omega$,
the dense subspace of $\eec_\omega$.
   \end{proof}

   The part \liczp 1 of above Theorem for a Hilbert space operator
valued $\omega$ is in \cite{sz_bull1}. It was long ago!
    \begin{com} \tlabel{t1.20.2}
   The presence of $\omega$ in the reproducing kernel Hilbert
C${}^*$--module, expressed in (\taggw), not only refreshes the
dilation formula \eqref{1.8.2} by giving it the simply looking
form \eqref{1z.8.2}
   but also works efficiently for minimality.
   \end{com}

\section*{Closer to the originals}
      \subsection*{KSGNS now} The door has been opened for
consequences. First KSGNS, for those who are eager for seeing it
again.

   Suppose $\sem$ and $\alga$ are \underbar{both}
C${}^*$--algebras; {the latter to be unital}. The map $\omega$ is
now \underbar{linear} and \underbar{completely}
\underbar{positive}. It is well known that
 this is the pleasant case when complete positivity is
\underbar{equivalent} to positive definiteness or rather to
$\alga$--positive definiteness according to the present
circumstances. In other words, positivity of an {\alga}--matrix,
at the very beginning of the exploration, is replaced by
factorization; this seems to be the only case, compare the items
(A) and (B) below.

   The boundedness condition (b) of Proposition \ref{t1.10.1}
follows immediately from the rudimentary inequality
$\semt^*\sems^*\sems\semt\Le\|\sems\|^2_\sem \,\semt^*\semt$.

Suppose $\{{\sf 1}_\lambda\}_\lambda$ be an approximate unit in
$\sem$. Insert in the Schwarz inequality \eqref{3.10.2}
$\sems_i={\sf 1}_\lambda$, $i\in\{{\rm singleton}\}$ and then
perform the limit passage (boundedness of $\omega$ in use) taking
advantage of all the attributes of $\{{\sf 1}_\lambda\}_\lambda$
so as to come to the second condition ($\lambda$) of Proposition
\ref{t2.8.2}. The first condition comes from \eqref{5.10.2} with
help of the approximative unit too.

Therefore both the \underbar{boundedness} \underbar{condition} and
the \underbar{extension} \underbar{property} are inherited from
the structural properties of C${}^*$--algebras.

   \begin{summ}
   If $\omega$ is a completely positive map between two
C${}^*$--algebras $\sem$ and $\alga$, with $\alga$ unital, then
\eqref{1.8.2} holds. This is the (ground level again) KSGNS, or
rather its \underbar{existence} part which is settled here in a
more elementary environment than usually. Its uniqueness is a
matter of further conditions unless $\sem$ is unital. Now we have
two possibilities: (a) either to assume there is an approximating
unit \underbar{in} $\sem$ which is a strong approximate unit
\underbar{for} $\omega$, or (b) to assume (\taggw) right away. The
mutual relationship is in Proposition \ref{t3.22.2}; an immediate
one is in the case $\sem=\alga$ and $\omega$ to be $\alga$--linear
they coincide.

   Notice that linearity of the dilated map $\varPhi$ is of
secondary importance, it comes for free from its minimality.
   \end{summ}

   What is meant by KSGNS refers to the case of $\alga$ to be a
C${}^*$--module which is C${}^*$--algebra in itself. Therefore,
our pre--KSGNS, so to speak, can be cultivated for this instance.

      \subsection*{Is the acronym KSGNS long enough?}

   Now we are going to provide arguments for our claim, or rather
insistence, of extending the acronym by the well-deserving
initials. It was 1955 when two essential events happened: the PAMS
paper \cite{stei} of Steinspring (included) and
Sz\H{o}kefalvi-Nagy's Appendix \cite{app} (forgotten); the English
version \cite{nagy} appeared five years later. Steinspring's paper
stimulated an explosion of new ideas, getting more and more
abstract, Sz.--Nagy's has been left out of favour.

   The first thing we want to stress on is these two 1955 results
are (logically) equivalent, see \cite{sz_nonl}. This is so as long
as their approaches are concentrated around bounded operators,
say. An attempt at extending them to the `unbounded' circumstances
causes a splitting into two, no longer equivalent:
   \begin{enumerate}
   \item[(A)] the direction of \cite{powers} in which
\underbar{complete} \underbar{positivity} further on provides with
abstract characterizations of
dilatability\,\footnote{\;Dilatability also means extendibility,
according to \S 5 of \cite{app}.} -- less useful in concrete
cases;
   \item[(B)] \underbar{positive} \underbar{definiteness} which
alone hardly becomes a sufficient condition for dilatability --
much more handy if available, look at \cite{ark} or \cite{nag_ext}
to catch some flavour.
   \end{enumerate}    The difference between (A) and (B) can be  even
seen in the case $\alga=\ccb$, that is when dealing with moment
problems.

   The main object in \cite{app} is an involution
semigroup\,\footnote{\;This notion seems to be originated there.}
(or, a $*$--semigroup) and positive definiteness is defined on
them. The positive definite mappings are bounded Hilbert space
operator valued which positions the Appendix before `K' in the
acronym; the mappings go to Hilbert $\ccb$--modules yet are
defined on simpler algebraic structures than C${}^*$--algebras.
The theorem says\,\footnote{\;This is in \S 6 of \cite{app}. Other
versions are in \cite{sz_bull}, \cite{sz_bull1}, \cite{sz_pams}
and \cite{ark}.} roughly that positive definiteness and the
boundedness condition (a) of Proposition \ref{t1.10.1} are
necessary and sufficient for an operator function to have a
dilation provided the $*$--semigroup $\sem$ has a unit. Addition
conclusions concern continuity and linearity properly understood
in the context of semigroups.

Let us exhibit the \underbar{diversity} of $*$--semigroups to
which the Sz.--Nagy general dilation theorem applies.
   \begin{enumerate}
   \item[$\triangleright$] \underbar{Groups} (commutative or not)
with involution $\sems^*\okr\sems^{-1}$. The boundedness condition
(a) of Proposition \ref{t1.10.1} turns into equality with
$c(\sems)=1$. Therefore the dilations are unitary representations
of groups, which happen in Harmonic Analysis. The case $\sem=\zzb$
is that of the 1953 frequently quoted Sz.--Nagy dilation theorem
for contractions, cf. \cite{app}, \S 4.
   \item[$\triangleright$] \underbar{Inverse}
\underbar{semigroups}. In this case the boundedness condition (d)
of Proposition \ref{t1.10.1} trivializes\,\footnote{\;Recall, to
$\sems\in\sem$ there is a unique $\sems^*$ such that
$\sems=\sems\sems^*\sems$ and $\sems^*=\sems^*\sems\sems^*$ },
which gives more prominence to Proposition \ref{t1.10.1}.
   \item[$\triangleright$] \underbar{$\sigma$--algebras} of sets
with set intersection as the semigroup operation and the identity
map as an involution. Again the boundedness condition (a) of
Proposition \ref{t1.10.1} turns into equality with $c(\sems)=1$.
This leads to Naimark's dilation of semispectral measures to
spectral ones, cf. \cite{app}, \S 2 and \cite{mlak}.
   \item[$\triangleright$] \underbar{Subnormality}. This is a 1950
invention of Halmos \cite{ha} who characterized it in terms of
positive definiteness plus some boundedness condition which, due
to Bram \cite{br}, turned out to be needless (see also
\cite{sz_pams} for another argument presented also in
\cite{nagy_rus}). It was Sz.--Nagy who put Halmos result into more
general framework, his general dilation theorem of \cite{app}, cf.
\S 5; another, quite different, kind of sentiment is exposed in
\cite{bill}. Nevertheless, subnormality can be considered also
from the C${}^*$--algebra point of view, see \cite{bunce} and
\cite{sz_pams1}. What is worthy to rescue is
 appearance of the unital $*$--semigroup $\nnb\times\nnb$ with
involution $(m,n)^*\okr(n,m)$ which makes all this possible.
   \item[$\triangleright$] \underbar{$*$--algebras}, in particular
\underbar{C${}^*$-algebras}. This case has been already discussed
on occasion of KSGSN.
   \item[$\triangleright$] \underbar{Moment} \underbar{problems}.
This is an extremely spectacular area rooted in Classical
Analysis. Here are three interesting cases \S 3
   \begin{itemize}
   \item $\sems\okr\nnb$ with identical involution. This is an
environment of the classical (one variable) moment problems, both
scalar and operator valued. Here positive definiteness is enough
for integral representation (read: dilatability). Any of the
boundedness conditions localizes the measure on a compact set.
   \item $\sems\okr\nnb^d$ with coordinate addition and identical
involution again. This is the very \underbar{sensitive} case, cf.
\cite{fu}, positive definiteness is only a necessary condition of
being a moment multisequence. Therefore, the boundedness condition
not only guarantees dilatability but also localizes the measure on
a compact set, again.
   \item The Sz.--Nagy semigroup mentioned on occasion of
subnormality is related to the complex moment problem, for more
consult \cite{polar}.
   \end{itemize}
   It seems to be needless to say that the above can be extended
to the Hilbert C${}^*$--module context as well.
   \end{enumerate}
   The careful reader has noticed that author's ambition here is
to \underbar{confront} complete positivity with positive
definiteness and to heighten awareness the latter suits more
situations in a rather elementary manner. So, the question
   turns up again: is the acronym KSGNS long enough? Maybe
   KS\colorbox{Light}{Sz.-N}GNS? The only thing against might be
   it, as a pictograph, to violate someone's aesthetical habits.
   Nevertheless, notice `Sz.' is from \colorbox{Light}{Sz.--Nagy}
   and seems to have nothing in common with the present author.

  \subsection*{Making it more spatial}
  Let us notice that for both Steinspring and Sz.--Nagy the common
target space for $\omega$ is the C${}^*$--algebra of bounded
operators on a Hilbert \underbar{space} which is a $\ccb$--module.
The difference is in the initial set $\sem$ for the kernel;
Sz.--Nagy's is the simplest possible, it does not bear any
unnecessary at the moment arrangement. On the other hand, for GNS
our choice of the target space in Theorem \ref{t1.24.2} fits in.
We may try to mend this incompleteness.

    Suppose $\sem$ is a $*$--semigroup and $\alga$ is a unital
C${}^*$--algebra; moreover, suppose $\omega$ satisfies (\taggw).

   \subsubsection*{Procedure \# 1}
   Fix a faithful $*$--representation $\pi$ of $\alga$ on a
Hilbert space $\hhc$. Then \eqref{1.3.3} can be written as an
equality for bounded operators on the Hilbert space $\hhc$
   \begin{equation*}
   \pi(\omega(\sems))=\pi(V)\gw\pi(\varPhi_\omega(\sems))\pi(V),
\quad \sems\in\sem
   \end{equation*}
   with $ \sems\mapsto{\pi(\varPhi_\omega(\sems))}$ being a
$*$--representation of $\sem$ on the Hilbert space $\hhc$.
   \subsubsection*{Procedure \# 2} Now $\eec$ is an arbitrary $\alga$--module.
For a ${\bf B}^*(\eec)$--function $\omega$ on $\sem$ we have a
kernel $K$ on $\sem\times\eec$ defined as
   \begin{equation*} %\label{1.2.3}
K(\sems,\xi,\semt,\eta)\okr\is{\omega(\sems^*\semt)\xi}{\eta}_\eec,\quad
\sems,\semt\in\sem,\;\,\xi,\eta\in\eec.
   \end{equation*}
   Then $K$ becomes an $\alga$--valued kernel which is
$\sem$--invariant. Recall ${\bf B}^*(\eec)$--positive definiteness
of $\omega$ means
   \begin{equation} \label{2.3.3}
\sum\nolimits_{i,j}T_i\gw \omega(\sems_i^*\sems_j)T_j\Ge0,
\quad(\sems_m)_M\subset\sem,\;\,(T_n)_N\subset {\bf B}^*(\eec).
   \end{equation}
   To show $K$ is $\alga$--positive definite
choose\,\footnote{\;It looks like we have to accept existence of
such a $\xi$ as what is often called a technical assumption.}
$\xi\in\eec$ such that $\is\xi\xi_\eec={\sf e}$ and define
$T_i\okr T_{\xi,\xi_i\algaa_i}$ as in \eqref{2.31.1}. Then
$T_i\in{\bf B}^*(\eec)$ and $T_i\xi=\xi_i\algaa_i$
   \begin{align*}
\sum\nolimits_{i,j}\algaa_i^*K(\sems_i,\xi_i,\sems_j,\xi_j)\algaa_j
=&\sum\nolimits_{i,j}\algaa_i^*
\is{\omega(\sems_i^*\sems_j)\xi_i}{\xi_j}_\eec\algaa_j
   \\=&
\sum\nolimits_{i,j}
\is{\omega(\sems_i^*\sems_j)\xi_i\algaa_i}{\xi_j\algaa_j}_\eec
   \\=&
\sum\nolimits_{i,j}
\is{\omega(\sems_i^*\sems_j)T_i\xi}{T_j\xi}_\eec
   {\stackrel{{\scriptscriptstyle{\mathsf{
\eqref{2.3.3}}}}}{\Ge}}0
   \end{align*}
    Therefore $K$ is $\alga$--positive definite provided $\omega$
is ${\bf B}^*(\eec)$--positive definite. We can now develop the
reproducing kernel routine for $K$ distinguishing the resulting
elements of the construction by the subscript ${}_\omega$.

Notice $\sem$ acts on $S=\sem\times\eec$ exclusively through its
first variable, the second remains untouched.

   To avoid getting involved in nonunital dispute, which would not
be a disaster due to our already worked out tools like Proposition
\ref{t3.22.2}, assume $\sem$ is unital. Assume also $K({\sf
1},\xi,{\sf 1},\eta)=\is \xi\eta_\eec$ which defines immediately
an isometry $\funk V\xi{K_{{\sf 1},\xi}}$. Moreover, $V\in{\bf
B}^*(\eec,\eec_\omega)$ with
$\funk{V\gw}{K_{\sems,\eta}}{\omega(\sems^*)\eta}$. Indeed,
   \begin{align*}
   \is{V\xi}{K_{\sems,\eta}}_{\eec_\omega}=\is{K_{{\sf
1},\xi}}{K_{\sems,\eta}}_{\eec_\omega}=
\is{\omega(\sems)\xi}{\eta}_\eec
=\is{\xi}{\omega(\sems^*)\eta}_\eec.
   \end{align*}
   Putting together most of what we have experienced so far,
especially Corollary \ref{t2.2.1}, we come to the yet another
dilation result, one which reminds more those of Sz.--Nagy and
Steinspring.
   \begin{thm} \tlabel{t1.3.3}
   Let the $*$--semigroup $\sem$ and the C${}^*$--algebra $\alga$
be unital. Suppose $\eec$ is an an inner product $\alga$--module
and $\omega$ is an ${\bf B}^*(\eec)$--function on $\sem$ such that
   \begin{equation*}
   \omega({\sf 1})={\sf 1}_\eec.
   \end{equation*}
   There exist an isometry $V\in{\bf B}^*(\eec,\eec_\omega)$ and a
$*$--representation $\varPhi_\omega$ of $\sem$ on $\eec_\omega$
such that
   \begin{equation*}
\is\xi{\omega(\sems)\eta}_\eec=\is{V\xi}{\varPhi_\omega(\sems)V\eta}_{\eec_\omega},
\quad \sems\in\sem,\;\,\xi,\eta\in\eec
   \end{equation*}
   if and only if $\omega$ is ${\bf B}^*(\eec)$--positive definite
and satisfies the boundedness condition.

   \noindent The triple $(\eec_\omega,\varPhi_\omega,V)$ remains
minimal.
   \end{thm}
   Theorem \ref{t1.3.3} is in \cite{itoh} exposed in a different
way.

   \subsubsection*{Advice}
   Combine Procedure \# 1 with Procedure \# 2 to get whatever is
possible. In particular, one may come closer to what is Section 3
of \cite{murphy}; let us mention that the reproducing kernel
construction in case of Hilbert C${}^*$--module valued kernels was
developed in \cite{itoh1}.

    \subsection*{Revitalizing  a question moved aside {\em or}
   a painful case} We have declared in \liczp 2 of Abstract to say
couple of words on whether a mathematical result survives or not
depends on an author rather than the result itself. This kind of
behaviour is in sharp contrast to unquestionable integrity of
Mathematics. The pretext for undertaking this `metamathematical'
topic, say, is what we have called the boundedness condition;
consult Proposition \ref{t1.10.1} for it and Remark \ref{t1.4.3}
for the source references. As we have had already pointed out its
importance is in the fact that it implies, via boundedness of the
dilation operators, some further properties like integrability in
the moment problem (and also as a kind of {\em bonus}: compactness
of the support of the representing measure). The story concerns
mainly condition (c) of Proposition \ref{t1.10.1} and goes as
follows. The condition (c) appeared in the 1977 paper
\cite{sz_pams} with $c$ ($\alpha$ there) to be submultiplicative.
Then, in 1984 two things happened: the paper \cite{b_m} and the
monograph \cite{berg_book} with $c$, called an absolute value
there. In \cite{b_m} $c$ satisfies a kind of `sub--C${}^*$'
condition, in \cite{berg_book}, p. 89, it is submultiplicative
plus some minor extra requirements. Both sources quote
\cite{sz_pams}, however, in \cite{b_m} the authors say on p. 167
`{\em Definition 1.2 is weaker than that given by Szafraniec ...}'
while in \cite{berg_book}, p. 141, one may find `{\em but somehow
similar conditions are implicit in Szafraniec (1977).}' Why
`weaker' or why 'implicit' is not clear, only the authors know. In
two notes quoted here as \cite{exposi} there is a thorough
discussion of 'weaker' put into a pretty much wider context,
interesting for itself too. The conclusion therein is this battle
is for nothing. On the other hand, in Proposition of
\cite{sz_pams} the condition in question is \underbar{explicitly}
stated as condition (ii). Does `implicit' mean `explicit' or the
other way? Nevertheless, so far it looks like it is nothing to
quarrel with but once the seeds has been already sowed all the
rest has been developing drastically. Among 14 papers reported in
MathSciNet as quoting \cite{b_m} only one, that of \cite{gl}, does
justice. The last of those fourteen \cite{resel} makes even the
title `politically correct'. How has it gone this way? Maybe
because already the 1987 survey paper \cite{berg} was quit of the
reference \cite{sz_pams}? Does it happen to be accidental?

   So why those bitter words are here? Certainly, because the
author has been waiting long enough for the people to get aware of
their sins. Also because these things happen more frequently than
one might think of and there is no forum for them to be weighed in
though sometimes, not too often, one may find journal notes
entitled {\em Acknowledgment of priority \ldots}

Dear PT Reader, please forgive us that if you find it
inappropriate!

   \section*{Some more author's personal remarks} In the vast
literature of the subject there certainly are better, further
going and mountainously more involved results than these here.
Being an outsider, under the pressure of time, the author has been
unable to penetrate all the writings; the owners of those results
are asked to forgive him any pain caused by not to be
quoted\,\footnote{\;The author (he is an outsider, remember
please) has just learned about \cite{f-g} where an unbelievable
number of bibliographical items has been classified, 1297 till the
12th of November 2008.}. However he does hope to resume the
search, always with \underbar{simplicity} and
\underbar{temperance} as a priority.

Getting
   older the author more
   and more understands and appreciates what his Master in
Mathematics, Professor Tadeusz Wa\.zewski \cite{waz}, used to mean
saying {\em parasite associations}. Unfortunately, they have been
spreading over and over pretty often driving the natural beauty
 out of Mathematics. This essay's intension is to demonstrate an
attempt at slowing down that overwhelming drift. As already
mentioned the author is determined to continue elsewhere his
efforts of clarifying the topic with emphasis on {\em what is
responsible for what}. As {\em organic} food or wine is
approaching everyday life the time comes for bringing this
environmental idea into Mathematics too.

   %%%%%%%%%%%%%%%%%%%%%%%%%%%%%%%%%%%%%%%%%%%%%%%%%%%%%%%%%%%%%%
   \bibliographystyle{amsplain}
   
   %%%%%
   \end{document}